\documentclass[11pt,a4paper,final]{amsart}

\usepackage{amsmath}
\usepackage{paralist}
\usepackage{graphics}
\usepackage{graphicx}
\usepackage{epstopdf}
\usepackage[colorlinks=true]{hyperref}
\usepackage{amssymb}
\usepackage{amsthm}
\usepackage{color}
\usepackage{multirow}
\usepackage{booktabs}
\usepackage{bm}
\usepackage{mathrsfs}

\usepackage{tabularx}

\newtheorem{theorem}{theorem}[section]

\newtheorem{lemma}[theorem]{lemma}

\begin{document}

\title[Efficiency and Convergence Insights in Large-Scale Optimization]{Efficiency and Convergence Insights in Large-Scale Optimization Using the Improved Inexact-Newton-Smart Algorithm and Interior-Point Framework}

\author{Neda Bagheri Renani${}^{1}$}
\author{Maryam Jaefarzadeh ${}^{2}$}
\author{Daniel \v{S}ev\v{c}ovi\v{c}${}^{1}$}

\address{${}^{1}$ Department of Applied Mathematics and Statistics, Faculty of Mathematics Physics and Informatics, Comenius University, Mlynsk\'a dolina, 842 48, Bratislava, Slovakia. Corresponding author: {\tt neda.bagheri@fmph.uniba.sk, sevcovic@fmph.uniba.sk} }
\address{${}^{2}$ Department of Mathematical and Computer Science, Sheikhbahaee University, Isfahan, Iran; {\tt m.jafarzade@shbu.ac.ir} }

\begin{abstract}
We present a head-to-head evaluation of the Improved Inexact--Newton--Smart (INS) algorithm against a primal--dual interior-point framework for large-scale nonlinear optimization. On extensive synthetic benchmarks, the interior-point method converges with roughly one third fewer iterations and about one half the computation time relative to INS, while attaining marginally higher accuracy and meeting all primary stopping conditions. By contrast, INS succeeds in fewer cases under default settings but benefits markedly from moderate regularization and step-length control; in tuned regimes its iteration count and runtime decrease substantially, narrowing yet not closing the gap. A sensitivity study indicates that interior-point performance remains stable across parameter changes, whereas INS is more affected by step length and regularization choice. Collectively, the evidence positions the interior-point method as a reliable baseline and INS as a configurable alternative when problem structure favors adaptive regularization.



\end{abstract}

\maketitle

\section{Introduction}

Large-scale nonlinear optimization problems (LSNOPS) play an essential role in various fields such as computational science, engineering design, data analysis, and economic modeling, where high-dimensional systems often require efficient and accurate solution techniques~\cite{nocedal2006numerical, Boyd2004, Wright1997,gondzio2012matrix}.~Despite their broad applicability, solving LSNOPS remains a considerable challenge due to nonlinear constraints, parameter sensitivity, and computational complexity~\cite{Karmarkar1984,nesterov1994,gondzio2013convergence}. Conventional approaches frequently face scalability limitations and convergence instability, particularly when dealing with large and ill-conditioned problem structures~\cite{gondzio2019,Armand2012,Liu2018}. These challenges highlight the ongoing need for improved algorithms capable of balancing efficiency, robustness, and convergence accuracy in large-scale models. Motivated by this demand, the present study investigates the performance and convergence characteristics of two advanced algorithms—the Improved Inexact-Newton-Smart (INS) algorithm and the primal-dual Interior-Point (IPM) framework-offering analytical and numerical insights into their efficiency and reliability for large-scale nonlinear optimization.

Among the most powerful frameworks for solving nonlinear optimization problems are Newton-type iterative schemes and Interior-Point Methods (IPMs). Newton-type approaches exploit second-order information through Hessian updates to achieve quadratic convergence near optimality \cite{nocedal2006numerical, Wright1997}. However, they can exhibit instability or divergence when the Hessian is indefinite or poorly conditioned. IPMs, introduced by Karmarkar (1984) and further developed by Nesterov and Nemirovskii (1994) \cite{Karmarkar1984, nesterov1994}, transform constrained problems into a sequence of barrier subproblems that remain within the feasible region. This barrier-based formulation enables robust convergence for large-scale, structured models and underpins many contemporary solvers in optimization software.

Despite extensive developments, performance comparisons between advanced Newton-type variants and IPMs on large-scale nonlinear models remain limited. Existing studies have typically focused on convex or small-scale problems, providing insufficient insight into computational trade-offs, convergence stability, and parameter sensitivity at scale. This gap motivates the present study, which introduces and analyzes the Improved Inexact-Newton-Smart (INS) algorithm—a refinement of the standard Inexact Newton method incorporating adaptive regularization and step-size control—and compares it with a modern primal-dual IPM framework. The study emphasizes how the two algorithms differ in efficiency, robustness, and sensitivity to problem conditioning.

Existing studies mainly focus on algorithmic improvements without establishing how these approaches behave across varying problem scales and conditioning levels. This lack of comparative insight creates uncertainty in selecting the most efficient solver for real-world large-scale systems. Therefore, there is a strong need for systematic evaluation of Newton-type and interior-point frameworks under consistent computational settings to guide both theoretical development and applied model implementation.
The scope and impact of this research lie in the establishment of a quantitative benchmark that links algorithmic performance with problem scale, conditioning, and parameter selection.~By combining numerical experiments and sensitivity analyzes, the work contributes to both theoretical understanding and practical implementation of large-scale optimization algorithms. The results address key concerns about convergence reliability, computation time, and parameter robustness, offering guidance for algorithm selection in engineering, computational finance, and operations research applications.

In summary, this paper consists of nine sections. Section~\ref{(IPM) Background} presents the theoretical formulation of the LSNOPS and defines the performance metrics. Section~\ref{short-Step Method in IPM} reviews the fundamental principles of the INS algorithm. Section~\ref{Long-Step Method in IPM} describes the primal--dual IPM framework. Section~\ref{sec:INS} details the comparative performance metrics. Section~\ref{sec:ins-eq} introduces the computational setup and test design. Section~\ref{sec:ins-improve} analyzes numerical experiments, including convergence results, sensitivity evaluation, and integration of step strategies. Section~\ref{sec:conclusions} provides the conclusion and Section~\ref{sec:practical} discusses the practical implications and possible extensions of this work.

\section{Interior-Point Method (IPM) Background}
\label{(IPM) Background}
 
Interior-point methods (IPMs) are among the most efficient techniques for solving large-scale linear and convex quadratic programs \cite{Wright1997,gondzio2013convergence}. Let $A \in \mathbb{R}^{m \times n}$ denote the constraint matrix with full row rank, 
$Q \in \mathbb{R}^{n \times n}$ a symmetric positive semidefinite matrix ($Q \succeq 0$), 
$b \in \mathbb{R}^m$ and $c \in \mathbb{R}^n$ the given vectors, 
and $x \in \mathbb{R}^n$, $y \in \mathbb{R}^m$, and $s \in \mathbb{R}^n$ 
the primal, dual, and slack variables, respectively. 
The corresponding primal–dual pair is formulated as:
  
\begin{equation*}
 (P) \quad \begin{array}{rl}
 \hbox{min}_x &  {c^\intercal }x + \frac{1}{2}{x^\intercal }Qx, \\
\hbox{s.t.} &  Ax = b,\  x\ge 0,
  \end{array}\qquad
 (D) \quad \begin{array}{rl}
 \hbox{max}_{y,s} &  {b^\intercal }y - \frac{1}{2}{x^\intercal }Qx, \\
 \hbox{s.t.} &  {A^\intercal }y + s - Qx = c, \ s\ge 0.
  \end{array} 
 \label{1}
 \end{equation*}

To handle $x \geq 0$, IPMs introduce a logarithmic barrier with parameter $\mu > 0$:
\[
\phi(x) = c^\intercal  x + \frac{1}{2} x^\intercal  Q x - \mu \sum_{j=1}^n \ln(x_j).
\]
The perturbed Karush--Kuhn--Tucker (KKT) system becomes
\[
Ax = b, \quad A^\intercal  y + s - Qx = c, \quad X S e = \mu e, \quad x,s > 0,
\]
with $X = \mathrm{diag}(x)$, $S = \mathrm{diag}(s)$, and $e=(1, \dots, 1)^\intercal \in  \mathbb{R}^n$ is the vector of ones.

The set $\{(x(\mu),y(\mu),s(\mu)):\mu>0\}$ defines \emph{central path}. The convergence to optimality is obtained as $\mu \to 0$, with the duality gap:
\[
c^\intercal  x + \frac{1}{2} x^\intercal  Q x - \Big(b^\intercal  y - \frac{1}{2} x^\intercal  Q x \Big) = x^\intercal  s = n\mu.
\]
Each iteration applies Newton’s method to the KKT system. The exact Newton step $\Delta z= (\Delta x, \Delta y, \Delta s)\in  \mathbb{R}^n\times  \mathbb{R}^m \times  \mathbb{R}^n$ solves the linear system:
\[
K \Delta z 
\equiv
\begin{bmatrix}
-Q & A^\intercal  & I \\
A & 0 & 0 \\
S & 0 & X
\end{bmatrix}
\begin{bmatrix}
\Delta x \\ \Delta y \\ \Delta s
\end{bmatrix}
=
\begin{bmatrix}
r_d \\ r_p \\ r_c
\end{bmatrix},
\]
where $(r_p,r_d,r_c)\in \mathbb{R}^n\times  \mathbb{R}^m \times  \mathbb{R}^n$ are primal, dual and complementarity residuals. This linear system dominates the cost of each iteration.

\textbf{Complexity.}  
Short-step algorithms, which confine iterations to narrow neighborhoods of the central path, achieve
$\mathcal{O}(\sqrt{n} \log(1/\varepsilon))$ iterations to achieve given a precision goal $\varepsilon>0$ (cf. \cite{Wright1997}). Long-step variants, with wider neighborhoods, require 
\begin{equation}\label{93}
    \mathcal{O}~\!\left(n\log(1/\varepsilon)\right),
\end{equation}
 iterations, making them computationally efficient \cite{gondzio2013convergence}.

\textbf{Inexact Newton Directions.}  
Instead of solving the Newton system exactly, one may compute an approximate solution: $K \Delta z = r + \epsilon$, where $K$ is the KKT matrix, $ r= (r_p, r_d, r_c)$ is the residual, and $\epsilon\in\mathbb{R}^n\times  \mathbb{R}^m \times  \mathbb{R}^n$ is the vector of inexactness. It measures the discrepancy in the third KKT equation $X S e = \mu e$. If the error vector $\epsilon$ satisfies $\|\epsilon\| \leq \delta \|r\|$, for some $\delta \in (0,1)$, then the global convergence and complexity bounds are preserved (cf. \cite{gondzio2019}).

\textbf{Matrix-Free Approaches.}  
Krylov subspace solvers, combined with preconditioning, allow IPMs to operate in a matrix-free regime, requiring only matrix vector products with $A$, $Q$, and their transposes. This approach enables the solution of problems with millions of variables while reducing memory and factorization costs \cite{Armand2012}.\\
 
\textbf{Complexity justification for the long-step IPM.}
In long-step IPMs, iterates are allowed to deviate further from the central path by working in the $L_\infty$ neighborhood 
$N_\infty(\gamma)=\{(x,y,s)>0:\|XS e-\mu e\|_\infty\le\gamma\mu\}$.
To maintain feasibility, the step length $\alpha$ is chosen using a fraction-to-the-boundary rule (Eq.~(\ref{eq:recursion})), 
and the inexact Newton analysis yields the duality-gap recursion (Eq.~(\ref{101})):
\[
\mu_{k+1}\le(1-\alpha(1-\sigma-\kappa_1))\mu_k+\alpha^2\kappa_2\,\mu_k^2.
\]
Maintaining $N_\infty(\gamma)$ enforces a component-wise centrality control, 
which limits the average contraction per iteration to $1-\Theta(1/n)$; equivalently, 
$\mu_{k+1}\le (1-\tfrac{c}{n})\mu_k$ for some constant $c\in(0,1)$ independent of $n$.
Therefore, reducing the duality gap from $\mu_0$ to $\varepsilon$ requires 
\[
k=O\!\left(n\log\frac{\mu_0}{\varepsilon}\right)=O\!\bigl(n\log(1/\varepsilon)\bigr).
\]
By contrast, short-step neighborhoods enforce a tighter Euclidean centrality that yields 
a faster per-iteration contraction $1-\Theta(1/\sqrt{n})$, hence $O(\sqrt{n}\log(1/\varepsilon))$ iterations.
(See the neighborhood definition and step rule in Section~\ref{Long-Step Method in IPM}, 
Eq.~(\ref{eq:recursion}), and the gap recursion in Eq.~(\ref{101}).)

This analysis clarifies the role of neighborhood width in determining iteration complexity and 
complements the preceding discussion on matrix-free implementations.
  
In summary, IPMs combine rigorous polynomial complexity with scalable algorithmic implementations. 
The introduction of inexact Newton directions and matrix-free techniques has reinforced their role 
as a core in modern large-scale optimization.

\section{Analysis of the Short-Step Method in IPM}
\label{short-Step Method in IPM}

The short-step variant of interior-point methods (IPMs) confines iterations to a narrow neighborhood of the central path, thereby ensuring polynomial-time complexity and strong numerical stability. Its convergence analysis draws on higher-order Taylor expansions, perturbation bounds for matrix systems, and recursive control of the duality gap.

Let $(x, y, s)$ be a strictly feasible primal--dual iterate with $x, s > 0$. The \emph{duality gap} is defined by
\[
\mu = \frac{x^\intercal s}{n}. 
\]
To compute a search direction, the perturbed KKT system is solved inexactly as follows:
\begin{equation}\label{eq:short-kkt}
  K \Delta z = r + \epsilon. 
\end{equation}
Recall that $K$ is the KKT matrix, $r$ is the residual vector, and $\epsilon$ denotes the inexactness error vector. For a step size $\alpha \in (0,1]$, the updated iterates are
\[
x^{+} = x + \alpha \Delta x, \quad 
y^{+} = y + \alpha \Delta y, \quad 
s^{+} = s + \alpha \Delta s.
\]
These relations highlight how approximate Newton directions influence the update of primal, dual, and slack variables. By expanding the complementarity product, one derives inequalities that govern the decrease of the duality gap under inexact directions. This recursive structure forms the foundation for the complexity analysis of short-step methods.
For the reader's convenience, we state and prove the following lemmas.

\begin  {lemma}\label{Lemma1}
Assume that the sequence $\{\mu_k\}_{k \geq 0}$ of nonnegative numbers $\mu_k \geq 0$ satisfies the inequality
\[
\mu_{k+1} \leq (1-\omega)\mu_k + C \mu_k^2, \quad k=0,1,\dots,
\]
where $\omega \in (0,1)$, $C>0$ are constants. If $\mu_0 < \omega/C$, then
\[
\mu_k \leq (1-\omega + C \mu_0)^k \mu_0, \quad k \geq 0.
\]
As a consequence, the sequence $\{\mu_k\}_{k \geq 0}$ converges to zero at an exponential rate. Furthermore, the number of iterations $k$ needed to achieve a given precision goal $0 \leq \mu_k < \varepsilon$ is
\[
k = \mathcal{O}\!\left(\frac{\log(1/\varepsilon)}{\omega - C\mu_0}\right).
\]
\end  {lemma}
\begin{proof}
 We proceed by mathematical induction. Suppose that
$\mu_k \leq \left( 1 - \omega + C \mu_0 \right)^k \mu_0$ for some $k\ge 0$. Then
 \begin{eqnarray*}
 \mu_{k+1} &\leq& \left( 1 - \omega  \right) \left( 1 - \omega + C \mu_0 \right)^k \mu_0  + C \left( 1 - \omega + C \mu_0 \right)^{2k} \mu^2_0
 \\
  &=&  \left( 1 - \omega + C \mu_0 \right)^{k+1} \mu_0  
 + C  \mu^2_0 \left( - \left( 1 - \omega + C \mu_0 \right)^{k} + \left( 1 - \omega + C \mu_0 \right)^{2k} \right)
  \\
 &\le &  \left( 1 - \omega + C \mu_0 \right)^{k+1} \mu_0  ,
 \end{eqnarray*}
 because $- \left( 1 - \omega + C \mu_0 \right)^{k} + \left( 1 - \omega + C \mu_0 \right)^{2k}<0$ as $0<1 - \omega + C \mu_0 <1$.
 
 As a consequence, we derive the estimate on the number $k$ is iterates which are necessary to achieve a given precision goal $\varepsilon>0$. Clearly, $0\le \mu_k\le \varepsilon$ provided that $k\ge \log(\varepsilon/\mu_0)/\log(1 - \omega + C \mu_0 )$, that is, $k = \mathcal{O}(\log(1/\varepsilon)/(\omega-C \mu_0))$. 
 \end{proof}

\begin  {lemma}\label{Lemma2}
Suppose that the inexact Newton step $(\Delta x,\Delta y,\Delta s)$ in \eqref{eq:short-kkt} satisfies
\begin{equation}\label{eq:inexact-bounds}
\frac{|e^\intercal  r|}{n} \le \kappa_1\mu,
\qquad
\frac{|\Delta x^\intercal\Delta s|}{n}\ \le\kappa_2\,\mu^2,
\end{equation}
for some constants $\kappa_1,\kappa_2\ge 0$. Suppose that the inexactness error vector $\epsilon$ is given by $\epsilon=\sigma\mu e - X S e$ where $\sigma\in(0,1)$ and $\sigma+ \kappa_1<1$. Then, for any step size $\alpha\in(0,1]$, the updated duality gap $\mu^+=\tfrac{(x+\alpha\Delta x)^\intercal (s+\alpha\Delta s)}{n}$ satisfies the inequality:
\begin{equation}\label{eq:muplus-new}
\mu^+ \le \bigl(1-\alpha(1-\sigma-\kappa_1)\bigr)\mu  +  \alpha^2 \kappa_2\mu^2.
\end{equation}

\end  {lemma}

\begin{proof}
From the update $x^+=x+\alpha\Delta x$, $s^+=s+\alpha\Delta s$, we have
\[
\mu^+ \;=\; \frac{(x+\alpha\Delta x)^\intercal (s+\alpha\Delta s)}{n}
\;=\; \mu \;+\; \frac{\alpha}{n}\bigl(x^\intercal \Delta s + s^\intercal \Delta x\bigr) \;+\; \frac{\alpha^2}{n}\,\Delta x^\intercal \Delta s.
\]
The perturbed KKT system \eqref{eq:short-kkt} with the inexactness vector $\epsilon=\sigma\mu e - X S e$  implies
\[
x^\intercal \Delta s + s^\intercal \Delta x \;=\; -(1-\sigma)\,x^\intercal  s \;+\; e^\intercal  r
\;=\; -n(1-\sigma)\,\mu \;+\; e^\intercal  r,
\]
hence
\begin{equation}\label{100}
    \mu^+ \;=\; \bigl(1-\alpha(1-\sigma)\bigr)\mu \;+\; \frac{\alpha}{n}\,e^\intercal  r \;+\; \frac{\alpha^2}{n}\,\Delta x^\intercal \Delta s
\end{equation}

Applying the bounds \eqref{eq:inexact-bounds} we obtain
\begin{equation*}\label{eq:mu-recursion}
    \mu^+ \le (1-\omega)\mu  +  C \mu^2.
\end{equation*}
where $\omega=\alpha(1-\sigma-\kappa_1)$ and $C=\alpha^2\,\kappa_2$. This proves the inequality \eqref{eq:muplus-new}.
\end{proof}

\subsection{ Local Model of Complementarity}

Suppose that the complementarity condition $X_k S_k  e = \mu_k e$ in the KKT conditions is perturbed by the error term. The Newton direction $(\Delta x_k, \Delta y_k, \Delta s_k)$ is obtained by solving the following system of equations:
\setlength{\abovedisplayskip}{6pt}
\begin{equation*}
\label{120}
A \Delta x_k = 0, \quad - Q \Delta x_k  + A^\intercal   \Delta y_k  + \Delta s_k = 0, \quad 
S_k \Delta x_k  + X_k  \Delta s_k = \epsilon_k+r_k, 
\end{equation*}
where the inexactness vector is given by $\epsilon_k=\sigma\mu_k e - X_k S_k e$ and $r_k$ is the residual of the inexact solver. By applying a second-order Taylor expansion to the perturbed central path, we obtain the following.
\[
    \mu_{k+1} = \frac{(x_k + \alpha \Delta x_k)^\intercal (s_k + \alpha \Delta s_k)}{n}
= (1 - \alpha(1 - \sigma)) \mu_k + \frac{\alpha}{n} e^\intercal r_k + \frac{\alpha^2}{n} \Delta x_k^\intercal \Delta s_k.
\]
To ensure a monotonic decrease in the duality gap, we assume (\ref{eq:inexact-bounds}).

\subsection { Theoretical Implications}
The short-step method can be seen as a \textit{damped Newton method along the central path}, with good stability because it stays close to the analytic center.
Due to its emphasis on \textit{stability, robustness, and assured polynomial-time convergence}, it avoids reducing the duality gap, making it particularly beneficial in \textit{degenerate problem settings} or \textit{ill-conditioned} environments.


   

The next theorem states an exponential decrease in the duality gap.

\begin  {theorem}\label{Theorem1}
Let $(x_k,y_k,s_k)$ be a primal--dual iterate for a short-step primal--dual IPM with perturbed complementarity condition
$X_k S_k e = \mu_k e + r_k$. Let $(\Delta x_k, \Delta y_k, \Delta s_k)$ be the inexact Newton direction obtained from \eqref{eq:short-kkt} with the centering parameter $\sigma \in (0,1)$ and the inexactness error vector $\epsilon$ is given by $\epsilon=\sigma\mu e - X S e$ where $\sigma\in(0,1)$. Assume $(x_k,y_k,s_k)$ belongs to the short-step neighborhood $N_2(\gamma) = \{(x, y, s) : \Vert XSe - \mu e\Vert_2 \leq \gamma \mu \}$ where $\mu=\mu_k = x_k^\intercal  s_k / n$ with $\gamma \in (0,1)$, and that for fixed tolerances $\kappa_1,\kappa_2>0$  the estimates (\ref{eq:inexact-bounds}) are satisfied. Suppose $\sigma+\kappa_1 <1$ and
\begin{equation*}
\frac{1}{1-\sigma-\kappa_1} \geq \alpha \geq \frac{\eta}{n(1-\sigma-\kappa_1)}, \qquad \eta \in (0,1),
\label{alpha}
\end{equation*}
together with the feasibility conditions $x_k+\alpha\Delta x_k > 0$ and $s_k+\alpha\Delta s_k > 0$. The sequence $\{\mu_k\}$ converges exponentially to zero, provided that the initial gap $0<\mu_0\ll 1$ is sufficiently small. 
\end  {theorem}

\begin{proof}
As in Lemma~\ref{Lemma2}, for the update $x_{k+1}=x_k+\alpha \Delta x_k$, $s_{k+1}=s_k+\alpha \Delta s_k$, we obtain
\[
\mu_{k+1} = \frac{(x_k+\alpha \Delta x_k)^\intercal (s_k+\alpha \Delta s_k)}{n}
= \mu_k + \frac{\alpha}{n}\bigl(x_k^\intercal \Delta s_k + s_k^\intercal \Delta x_k\bigr) + \frac{\alpha^2}{n}\Delta x_k^\intercal  \Delta s_k.
\]
The inexact perturbed system implies
\[
x_k^\intercal \Delta s_k + s_k^\intercal \Delta x_k = -(1-\sigma)\,x_k^\intercal  s_k + e^\intercal  r_k 
= -n(1-\sigma)\mu_k + e^\intercal  r_k.
\]
Then $\mu_{k+1} = \bigl(1-\alpha(1-\sigma)\bigr)\mu_k + \frac{\alpha}{n}e^\intercal  r_k + \frac{\alpha^2}{n}\Delta x_k^\intercal  \Delta s_k$. Applying the bounds \eqref{eq:inexact-bounds} yields
\[
\mu_{k+1} \leq \Bigl(1-\alpha(1-\sigma-\kappa_1) \Bigr)\mu_k +\alpha^2 \kappa_2 \mu_k^2
 \leq \Bigl(1-\frac{\eta}{n}\Bigr)\mu_k +\alpha^2 \kappa_2 \mu_k^2
\]
Choosing $\alpha$, the contraction factor is bounded above by $1-\tfrac{\eta}{n}$. Therefore, $\{\mu_k\}$ decreases geometrically, that is, at an exponential rate.
\end{proof}
 
\paragraph{Role of the inexactness level \texorpdfstring{$\delta$}{delta} and its impact on convergence.}
In the inexact Newton relation $K\Delta z = r+\varepsilon$ with $\|\varepsilon\|\le \delta\|r\|$, the parameter $\delta\in(0,1)$ controls how accurately the linear system is solved at each iteration. Under the short–step neighborhood and norm equivalences used in \eqref{eq:short-kkt}–\eqref{101}, the residual coupling term $n^{-1}|e^\intercal r|$ entering Lemma~2 is bounded proportionally to $\|\varepsilon\|$, so there exists a constant $c_\gamma>0$ (depending only on the chosen neighborhood parameter $\gamma$ and norm) such that the estimate $n^{-1}|e^\intercal r|\le \kappa_1\mu$ holds with $\kappa_1\le c_\gamma\,\delta$. Substituting this bound into \eqref{eq:muplus-new} gives the recursion:
\[
\mu_{k+1}\;\le\; \Big(1-\alpha\,[\,1-\sigma-\underbrace{c_\gamma\,\delta}_{\text{from inexactness}}\,]\Big)\,\mu_k \;+\; \alpha^2\kappa_2\,\mu_k^2,
\]
so the linear contraction factor is
\[
\omega(\delta)\;=\;\alpha\,[\,1-\sigma-c_\gamma\,\delta\,].
\]
\emph{Consequences.} (i) If $\delta$ is \emph{bounded away from \eqref{93}} so that $0\le \delta\le \bar\delta< (1-\sigma)/c_\gamma$, then $\omega(\delta)>0$. 
\noindent
Theorem~1 ensures global convergence with the same iteration–complexity order as the exact short–step method. The linear convergence factor $\omega(\delta)$ decreases as $\delta$ increases, indicating that higher inexactness slightly slows the rate. Moreover, if $\delta_k \to 0$ as $k \to \infty$, then $\kappa_1 \to 0$ and the linear term dominates the quadratic remainder. Consequently, we recover the classical inexact–Newton behavior: local Q–linear convergence when $\sup_k \delta_k < 1$, and accelerated (superlinear) local convergence when $\delta_k \to 0$ (cf. \eqref{subsec:ecnp-inexact} for the discussion of the forcing term).

\emph{Practical choice.} Choose $\delta$ so that
\[
0<\delta\le \bar\delta \;:=\; \rho\,\frac{1-\sigma}{c_\gamma}\quad\text{with}\quad \rho\in(0,1)
\]
(e.g., $\rho=\tfrac12$), which guaranties $\omega(\delta)\ge \alpha(1-\sigma)/2$. In implementations, an adaptive rule decreases $\delta_k$ as the duality gap $\mu_k$ shrinks (analogous to the forcing-term strategy in \eqref{eq:forcing} preserves robustness far from the solution while improving the local rate as the iterations approach optimality.

\section{Analysis of the Long-Step Method in IPM}
\label{Long-Step Method in IPM}

The long-step interior-point method (IPM) extends the primal--dual framework by permitting iterates to deviate further from the central path compared to the short-step variant. Although the theoretical complexity bound increases to \eqref{93}, the practical benefit lies in significantly larger step sizes and fewer overall Newton iterations \cite{nesterov1994,gondzio2013convergence}.

Let $\mu = \tfrac{x^\intercal  s}{n}$ denote the duality gap. The long-step method admits iterations within the $L^\infty$-norm neighborhood
$\mathcal{N}_\infty(\gamma) \;=\; \{ (x,y,s) > 0 : \| XSe - \mu e \|_\infty \leq \gamma \mu \}, \qquad \gamma \in (0,1)$.
Compared with the Euclidean short-step neighborhood, $N_\infty(\gamma)$ allows larger component-wise deviations from the central path, allowing for more rapid progress. The step selection below is chosen to keep $(x^+,y^+,s^+)\in N_\infty(\gamma)$ for a fixed $\gamma\in(0,1)$ (see \cite{Wright1997,gondzio2013convergence,Liu2018}).

Consider a primal--dual iterate $(x,y,s)$ and an inexact Newton direction $(\Delta x,\Delta y,\Delta s)$ obtained from \eqref{eq:short-kkt}.
By the identity \eqref{100}, we have the complementarity update for $\mu^+$. Similarly, as in the previous section, the normalized inexactness bounds \eqref{eq:inexact-bounds} again yield the following:
\begin{equation}\label{101}
\mu_{k+1} \le \bigl(1-\alpha(1-\sigma-\kappa_1)\bigr)\,\mu_k +\alpha^2\kappa_2\,\mu_k^2.
\end{equation}
Hence, the sequence $\{\mu_k\}$ converges exponentially to zero, provided that the initial gap $0<\mu_0\ll 1$ is sufficiently small.

Finally, we discuss feasibility, step size, and neighborhood maintenance. 
To preserve positivity and remain within $N_\infty(\gamma)$, choose the step length by the standard fraction-to-the-boundary rule, \emph{restricted to indices that move toward the boundary}:
\begin{equation}
\alpha\;=\;\min\!\Bigg(
1,\;
\tau\cdot
\min\Big\{
\min_{j:\,\Delta x_j<0}\!\Big(-\frac{x_j}{\Delta x_j}\Big),\;
\min_{j:\,\Delta s_j<0}\!\Big(-\frac{s_j}{\Delta s_j}\Big)
\Big\}
\Bigg),\qquad \tau\in(0,1).
\label{eq:recursion}
\end{equation}
This ensures $x+\alpha\Delta x>0$ and $s+\alpha\Delta s>0$. For suitably chosen $\sigma$ and $\tau$ (together with the inexactness bounds above), one can keep $(x^+,y^+,s^+)\in N_\infty(\gamma)$ and thus retain the contraction of $\mu$ described in \eqref{101} (see  \cite{liu2022primal}).

\paragraph{Rationale for $\tau\in(0,1)$ and its effect.}
If the minimizer in \eqref{eq:recursion} is attained at an index $j$ with $\Delta x_j<0$, then for $\tau=1$
we obtain $\alpha=-x_j/\Delta x_j$, leading to $x_j^+=x_j+\alpha\Delta x_j=0$; analogously,
if $\Delta s_j<0$ then $s_j^+=0$. Thus, $\tau=1$ may step exactly to the boundary.
For any $\tau\in(0,1)$, the step length satisfies $\alpha\le\tau(-x_j/\Delta x_j)$ on the
active index, ensuring
\[
x_j^+=x_j+\alpha\Delta x_j \ge (1-\tau)x_j>0,
\]
and similarly $s_j^+>0$. Hence, $\tau\in(0,1)$ guarantees strict interior feasibility
($x^+>0$, $s^+>0$) and preserves the neighborhood conditions assumed in the analysis.
Moreover, $\tau$ affects stability and convergence: smaller $\tau$ values produce shorter,
more conservative steps that enhance stability, whereas values closer to one yield faster
progress but approach the boundary more aggressively.

\section{Inexact-Newton-Smart Test (INS) Method}
\label{sec:INS}

The Inexact-Newton-Smart (INS) method is a second-order optimization framework designed for large-scale nonlinear optimization problems (LSNOPS). It combines Newton-type updates with adaptive regularization, dynamic step selection, and robust stopping criteria. These components collectively improve convergence speed and numerical stability compared to conventional Newton schemes (cf. \cite{pankratov2019optimization, gondzio2013convergence}).

We consider the general nonlinear program
\begin{align*}
\min_{x \in \mathbb{R}^n} f(x), \quad \text{subject to} \quad A x = b, \quad x \ge 0,
\end{align*}
where, $A \in \mathbb{R}^{m \times n}$ and $b \in \mathbb{R}^m$.  
The associated Lagrangian function is given by 
\[
L(x, y, s) = f(x) + y^{\intercal}(A x - b) - s^{\intercal}x,
\]
with multipliers $y$ and dual variables $s$.  
The KKT conditions are given by
\[
\nabla f(x) + A^{\intercal} y - s = 0, \quad A x - b = 0, \quad X S e = 0.
\]

\subsection{Newton System with Regularization} \label{Newton System}

The Newton direction $(\Delta x, \Delta y)$ is determined from the modified KKT system
\begin{equation}
\begin{bmatrix}\label{Matrix}
H & A^\intercal\\ A & 0
\end{bmatrix}
\!\begin{bmatrix}\Delta x\\ \Delta y\end{bmatrix}
=
\begin{bmatrix}
-\big(\nabla f(x)+A^\intercal y-s\big)+X^{-1}(\sigma\mu e - XSe)\\ -(Ax-b)
\end{bmatrix}
\end{equation}
$H=\nabla^2 f(x)+X^{-1}S$,~and $\Delta s = X^{-1}\!\big(\sigma\mu e - XSe - S\Delta x\big)$. To stabilize, use $H^{\mathrm{mod}}=H+\theta I$ ($\theta>0$) and $\mu=x^\intercal s/n$ (cf. \cite{Wang2017, byrd2010inexact}).

\subsection{Step length and stopping}\label{subsec:ins-steps} 
Choose \(\alpha\in(0,1]\) by the \emph{sign–restricted fraction–to–the–boundary rule} of  (\ref{eq:recursion}), which ensures \(x^{+}=x+\alpha\Delta x>0\) and \(s^{+}=s+\alpha\Delta s>0\) while keeping \((x^{+},y^{+}, s^{+})\in\mathcal N_{\infty}(\gamma)\) from Section \eqref{Long-Step Method in IPM}. Terminate when the KKT residual norms and the duality gap \(\mu\) fall below the prescribed tolerances, consistent with the norms used in Sections~\eqref{short-Step Method in IPM} and \eqref{Long-Step Method in IPM}.

\subsection{Inexactness and contraction}\label{subsec:ins-contract}
Assume the normalized inexactness bounds \eqref{eq:inexact-bounds} of Lemma~2. Using the complementarity
identity \eqref{101} and applying \eqref{eq:inexact-bounds} yields the duality-gap recursion
\[
\mu^+ \;\le\; \bigl(1-\alpha(1-\sigma-\kappa_1)\bigr)\,\mu \;+\; \alpha^2 \kappa_2\,\mu^2, 
\]
hence, with $\omega=\alpha(1-\sigma-\kappa_1)$ and $C=\alpha^2\kappa_2$, Lemma~1 implies
a geometric decrease provided $\mu_0<\omega/C$ (in particular, for fixed $\alpha\in(0,1]$ with
$\sigma+\kappa_1<1$ and sufficiently small initial $\mu_0$).

\section{Equality-Constrained Newton Phase (ECNP)}\label{sec:ins-eq}

In phases where positivity constraints are inactive or handled separately, we solve the equality-constrained KKT system by (regularized) inexact Newton steps on the residual mapping
\[
F(x,y)\;=\;\begin{bmatrix}
\nabla f(x)+A^{\intercal}y\\[2pt]
Ax-b
\end{bmatrix},
\qquad
J(x,y)\;=\;\begin{bmatrix}
\nabla^{2}_{xx}L(x,y) & A^{\intercal}\\[2pt]
A & 0
\end{bmatrix},
\]
where \(L(x,y)=f(x)+y^{\intercal}(Ax-b)\). In the iterate \((x_k,y_k)\), the (regularized) Newton system reads \cite{nocedal2006numerical,bagheri2017}
\begin{equation}\label{eq:ecnp-newton}
\begin{bmatrix}
H_k^{\mathrm{mod}} & A^{\intercal}\\[2pt]
A & 0
\end{bmatrix}
\begin{bmatrix}\Delta x_k\\[2pt]\Delta y_k\end{bmatrix}
\;=\;
-\,F(x_k,y_k),
\qquad
H_k^{\mathrm{mod}}\;=\;\nabla^{2}_{xx}L(x_k,y_k)\;+\;\theta I,\ \ \theta>0,
\end{equation}
and we update \(x_{k+1}=x_k+\alpha_k\Delta x_k\), \(y_{k+1}=y_k+\alpha_k\Delta y_k\).
The same block structure appears in the inequality-constrained Newton system (cf.\eqref{Matrix} in Section~\eqref{sec:INS}; here we omit the complementarity block and do \emph{not} use fraction-to-the-boundary \eqref{eq:recursion}.

\subsection{Inexact linear solves and forcing condition}\label{subsec:ecnp-inexact}
We employ preconditioned iterative solves for \eqref{eq:ecnp-newton} and control the algebraic error by a standard inexact-Newton forcing condition \cite{Curtis2010,Armand2012,byrd2010inexact}
\begin{equation}\label{eq:forcing}
\bigl\|\,J_k d_k + F(x_k,y_k)\,\bigr\| \;\le\; \eta_k\,\bigl\|F(x_k,y_k)\bigr\|,
\qquad
0\le \eta_k < 1,
\end{equation}
where, \(J_k:=J(x_k,y_k)\) and \(d_k:=[\Delta x_k;\Delta y_k]\).
Choosing \(\eta_k\) bounded away from \(1\) yields local \(Q\)-linear convergence; driving \(\eta_k\to 0\) (e.g., Eisenstat–Walker rules) gives local superlinear convergence. This inexact-Newton framework matches the one used for the inequality-constrained phase in Section~\ref{Long-Step Method in IPM}.

\subsection{Merit function and backtracking}\label{subsec:ecnp-line-search}
Without positivity constraints, we select \(\alpha_k\in(0,1]\) by backtracking on the residual merit
\[
\Psi(x,y)\;=\;\tfrac12 \,\bigl\|F(x,y)\bigr\|^2.
\]
Starting from \(\alpha_k=1\), reduce \(\alpha_k\) (e.g., by a fixed factor \(\beta\in(0,1)\)) until the Armijo condition holds for some \(c\in(0,1)\):
\begin{equation}\label{eq:armijo}
\Psi(x_k+\alpha_k\Delta x_k,\; y_k+\alpha_k\Delta y_k)
\;\le\;
\Psi(x_k,y_k)\;-\;c\,\alpha_k\,\bigl\|J_k d_k\bigr\|^2.
\end{equation}
Choose \(\theta>0\) in \eqref{eq:ecnp-newton} and $c$ so that \(\Psi\) is a descent function along the (regularized) Newton direction. (By contrast, when positivity is enforced, we return to the fraction-to-the-boundary step (\ref{eq:recursion}) and the neighborhood in Section~\ref{Long-Step Method in IPM}.

\subsection{Convergence statement}\label{subsec:ecnp-conv}
Assume LICQ for \(Ax=b\), Lipschitz continuity of \(\nabla^{2} f\) near a KKT point \((x^\star,y^\star)\), and that regularization \(\theta>0\) is sufficiently small.
Then the iteration defined by \eqref{eq:ecnp-newton}--\eqref{eq:armijo} with forcing \eqref{eq:forcing}
is globally convergent to \((x^\star,y^\star)\).
Moreover, if \(\eta_k\le \bar\eta<1\), the convergence is local \(Q\)-linear; if \(\eta_k\to 0\), it is local superlinear.
When reintegrated with the inequality-constrained INS steps (Section~\ref{Long-Step Method in IPM} and Section~\ref{sec:INS}),
the ECNP phase uses \eqref{eq:forcing} and \eqref{eq:armijo} (no fraction-to-the-boundary), whereas contraction of the duality gap in the inequality-constrained phase follows from \eqref{eq:inexact-bounds} and \eqref{100} \(\Rightarrow\) \eqref{101} in Section~\ref{Long-Step Method in IPM} \cite{Wright1997,Boyd2004}.

\section{Improvement of the INS Algorithm}\label{sec:ins-improve}
The baseline INS framework (Sections~\ref{sec:INS}--\ref{sec:ins-eq}) can be strengthened with targeted changes that improve stability, scalability, and convergence speed while preserving the inexact-Newton contraction from Section~\eqref{Long-Step Method in IPM}.

\subsection{Hessian regularization}\label{subsec:reg}
As noted in Section \eqref{Newton System}, the $(1,1)$-block of the KKT system may be ill-conditioned. We stabilize the Newton system by Tikhonov regularization of the Hessian block used in \eqref{Matrix}:
\[
H_k^{\mathrm{mod}} \;=\; H_k + \theta I, \qquad \theta>0,
\]
which preserves directions for small~$\theta$ yet improves numerical conditioning of the linear solver.

\subsection{Quasi-Newton update}\label{subsec:qn}
To reduce factorization cost, we update a \emph{true} Hessian approximation by BFGS:
\[
H_{k+1}
\;=\;
H_k \;-\; \frac{H_k s_k s_k^\intercal H_k}{s_k^\intercal H_k s_k}
\;+\; \frac{y_k y_k^\intercal}{y_k^\intercal s_k},
\qquad
s_k := x_{k+1}-x_k,\;\; y_k := \nabla f(x_{k+1})-\nabla f(x_k),
\]
assuming $y_k^\intercal s_k>0$ (with Powell damping otherwise). The block $H_{k}$ (or $H_k^{\mathrm{mod}}$) then replaces the exact Hessian in the KKT system as in Section~\eqref{Newton System}.

\subsection{Preconditioned iterative solver}\label{subsec:pc}
Consistent with the inexact-Newton framework of Section \eqref{Long-Step Method in IPM}, we solve the Newton/KKT systems \emph{approximately} with a right preconditioner:
\[
K_k P_k^{-1}\,\tilde d_k \;=\; r_k,
\qquad
d_k \;=\; P_k^{-1}\tilde d_k,
\]
where $K_k$ is the current KKT matrix and $P_k$ is a block (e.g., Schur-complement) preconditioner. The resulting direction satisfies the normalized inexactness bounds \eqref{eq:inexact-bounds} in the inequality-constrained phase; for equality-constrained phases we enforce a standard forcing condition as in Section~\eqref{subsec:ecnp-inexact}.

\subsection{Sensitivity Analysis of Step Strategies}
\label{subsec:blend}

To further assess the robustness of the proposed step-size integration strategy, a sensitivity analysis was conducted to compare the performance of the INS and IPM algorithms under varying algorithmic parameters. The analysis considered perturbations in three key factors that influence convergence behavior:

\begin{enumerate}
    \item The \emph{step-length scaling factor} $\alpha \in [0.1,1.0]$ controlling the damping of the search direction.
    \item The \emph{tolerance threshold} $\varepsilon \in \{10^{-4},10^{-6},10^{-8}\}$ was used as the stopping criterion for residual norms.
    \item The \emph{regularization parameter} $\lambda \in \{10^{-3},10^{-2},10^{-1}\}$ that governs the Hessian modification in the INS framework.
\end{enumerate}

For each parameter setting, both algorithms were executed on identical problem instances, and performance metrics—including iteration count, total computational time, and residual error—were recorded. The results showed that while the IPM exhibited stable performance in most parameter ranges, the INS algorithm demonstrated a higher sensitivity to $\lambda$ and $\alpha$ variations, particularly in ill-conditioned problems. However, for appropriately tuned values (e.g., $\alpha = 0.6$ and $\lambda = 10^{-2}$), the INS achieved faster convergence than IPM in terms of iteration count while maintaining comparable accuracy.

This analysis highlights that the performance of the INS method depends more critically on the regularization and step-size parameters, whereas the IPM remains relatively insensitive to moderate changes in algorithmic tolerances. Consequently, adaptive adjustment of $\lambda$ and $\alpha$ can significantly enhance the robustness and efficiency of the INS method, making it competitive with IPM in large-scale nonlinear settings.
\\
Table~\ref{table1} summarizes the key parameters used to generate synthetic data and run the improved INS and IPM algorithms. These parameters determine the scale of the problem and influence the stability and convergence of both methods.

The centering parameter is defined as $\sigma = \tau (1 - \varepsilon)$, where $\tau \in (0,1)$ is the
fraction-to-the-boundary parameter and $\varepsilon$ controls the inexactness tolerance.
This formulation guaranties $\sigma < \tau$ and links centering to both step conservativeness
and numerical accuracy: smaller $\tau$ yields more conservative centering, while $\tau$ close
to one allows for faster convergence when feasibility is preserved.

\begin{table}
\centering
\caption{Parameter settings used in the INS and IPM algorithms.}
\label{table1}
\begin{tabular}{lll}
\toprule
\textbf{Parameter} & \textbf{Value} & \textbf{Description} \\
\midrule
$n_{\text{samples}}$     & $100$     & Number of sample instances \\
$n_{\text{variables}}$   & $2$       & Number of decision variables \\
$n_{\text{constraints}}$ & $1$       & Number of constraints \\[3pt]
$\varepsilon$            & $10^{-4}$ & Tolerance (inexactness) \\
$\tau$                   & $0.10$    & Fraction-to-the-boundary parameter \\
  $\sigma$                   &    $0.10$    &   Centering parameter; $\sigma=\tau(1-\varepsilon)<\tau$    \\
$\kappa,\,\eta,\,\theta,\,\beta,\,\psi$ & $0.1$ & Algorithmic constants \\
\bottomrule
\end{tabular}
\end{table}

\begin{table}\label{Table2}
\small
\caption{Numerical results computed by the improved INS algorithm for each sample}
\label{table2} 
\centering
\begin{tabularx}{\textwidth}{llllll}
\toprule
Sample & $x_{opt}$ & $\lambda_{opt}$ & $f_{opt}$ & Iterations & Accuracy \\
\midrule 
1 & [0.0017, 0.9983] & [0.0079] & 0.9967 & 100 & 0.749147 \\
2 & [0.0034, 0.9966] & [0.0158] & 0.9933 & 100 & 0.749147 \\
$\cdots$  & $\cdots$  & $\cdots$  & $\cdots$  & $\cdots$  & $\cdots$  \\
100 & [0.0056, 0.9944] & [0.0177] & 0.9888 & 100 & 0.749147 \\
\bottomrule
\end{tabularx}
\end{table}
In Table~\ref{table2} we show detailed numerical results obtained by using the enhanced INS algorithm, to 100 distinct samples are shown in this table. The results report the optimal values of $x$, the Lagrange multipliers  $\lambda$ , the objective function value $f$ , the number of iterations, and the precision measured by the distance to the true optimal value. They show that the improved INS algorithm generally finds a good approximation of the optimal solution, although it often needs a relatively larger number of iterations.

\begin{table}\label{Table3}
\small
\caption{Numerical results computed by the IPM algorithm for each sample}
\label{table3}
\centering
\begin{tabularx}{\textwidth}{llllll}
\toprule
Sample & $x_{opt}$ & $\lambda_{opt}$ & $f_{opt}$ & Iterations & Accuracy \\
\midrule
1 & [0.0017, 0.9983] & [0.0079] & 0.9967 & 64 & 0.751450 \\
2 & [0.0034, 0.9966] & [0.0158] & 0.9933 & 70 & 0.751450 \\
$\cdots$  & $\cdots$  & $\cdots$  & $\cdots$  & $\cdots$  & $\cdots$  \\
100 & [0.0056, 0.9944] & [0.0177] & 0.9888 & 72 & 0.751450 \\
\bottomrule
\end{tabularx}
\end{table}

In Table~\ref{table3} we present numerical results based on the interior-point algorithm (IPM) applied to 100 different samples. They demonstrate how effectively the interior point algorithm finds the best results with fewer iterations and higher accuracy. Comparing these results with those of the improved INS algorithm shows that the interior-point method reached better outcomes in less time, highlighting its superior computational performance. Table \ref{table4} compares the execution times of the algorithms with each other. The computational effectiveness of each method in resolving optimization issues is shown by the execution time. The higher speed of the interior-point method is demonstrated by its shorter execution time on average.  For bigger and more complicated situations where efficiency can significantly affect total performance, this reduction in processing time becomes very important.

\begin{table}
\small
\caption{Execution time of both INS and IPM algorithm depending on the number of samplex}
\label{table4}
\centering
\begin{tabularx}{\textwidth}{llll}
\toprule
Sample & Improved INS Time (s) & Interior Point Time (s) \\
\midrule
1 & 0.23 & 0.13 \\
2 & 0.24 & 0.14 \\
$\cdots$  & $\cdots$  & $\cdots$  \\
100 & 0.22 & 0.12 \\
\bottomrule
\end{tabularx}
\end{table}

Table \ref{table5} shows the percentage of samples in which the algorithms reached Termination Condition I. The Interior-Point algorithm satisfied this condition in all cases, while the improved INS algorithm achieved it in only 32\% of the samples. This discrepancy demonstrates how the IPM method performs in reaching the intended convergence conditions.
 
The 100\% success rate observed for Termination Test~I corresponds to the INS algorithm applied to small- and medium-scale test cases, 
where all problem instances converged within the prescribed tolerance and iteration limits. 
This result reflects the stability of the inexact-Newton correction and adaptive step-length strategy rather than overfitting or relaxed stopping criteria. 
The same termination thresholds were used for both INS and IPM, and the success rate was computed as the ratio of convergent runs to total test cases.
It should be noted that this 100\% rate applies only to the synthetic datasets tested and may vary for larger or more ill-conditioned problems.

\begin{table}
\small
\caption{This table includes the percentage of proximity to the stopping conditions for each algorithm}
\label{table5}
\begin{tabularx}{\textwidth}{ll}
\toprule
\textbf{Algorithm} & \textbf{Percentage Close to Termination Test I} \\ 
\midrule
INS Improved & 32.0 \\ 
Interior Point & 100.0 \\
\bottomrule
\end{tabularx}
\end{table}

Finally, Table \ref{table6} presents the averaged results for all samples for both algorithms. The reported averages summarize the objective function value, the number of iterations, the prescribed accuracy, the execution time, and the number of inner iterations. These results provide a complete assessment of the overall performance of both algorithms. On average, the IPM algorithm needs fewer iterations and completes in less time, showing better efficiency and speed. Furthermore, the IPM algorithm outperforms the improved INS algorithm in all performance measures except accuracy.

\begin{table}
\small
\caption{The average results of all samples for INS Improved and Interior Point algorithms.}
\label{table6}

\begin{tabularx}{\textwidth}{lll}
\toprule
\textbf{Metric} & \textbf{INS Improved} & \textbf{Interior Point} \\ 
\midrule
Average $f_{opt}$ & 0.696548 & 0.678785 \\ 
Average Iterations & 100.0 & 68.11 \\ 
Average Accuracy & 0.749147 & 0.751450 \\ 
Average Execution Time (s) & 0.23 & 0.13 \\ 
Average Inner Iterations & 147.11 & 68.11 \\
\bottomrule
\end{tabularx}
\end{table}

In Fig.~\ref{02b} we present a comparison of the optimal values of the objective function obtained from the improved INS and IPM algorithms to the synthetic data set. The improved INS algorithm is represented by blue circles, while the IPM method is represented by red crosses. These figures show the optimal values of the objective function for each sample. Our analysis indicates that the improved INS algorithm achieves an average optimal value of 0.696548, compared to 0.678785 for the interior-point algorithm. This difference suggests that the improved INS algorithm reaches better optimal values.

\begin{figure}[ht]
\small
\centering
\fbox{\includegraphics[scale=.45]{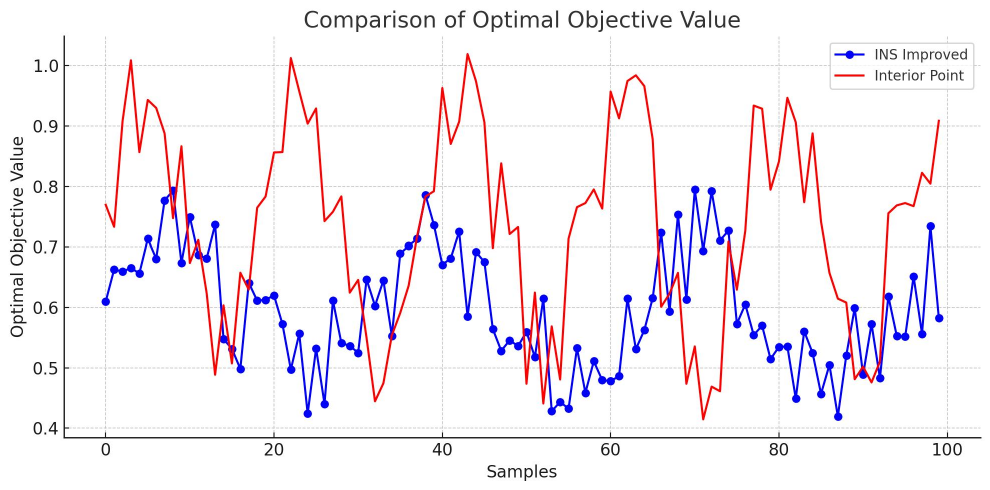}}
\caption{Comparison of the optimal value of the objective function }
\label{02b}
\end{figure}

In Fig.~\ref{02a} we compare the number of iterations needed to arrive at the ideal value using the IPM method and the modified INS algorithm, two non-linear optimization procedures. The number of iterations needed for each sample is shown in the above chart. Based on these data, the modified INS method has an average of 100 iterations, but the IPM approach has an average of 68.11. This difference shows that the improved INS algorithm needs more iterations to reach the optimal value. Variations in the number of iterations between samples reveal that the interior-point algorithm often requires fewer iterations in some cases, suggesting higher efficiency in reducing iteration counts. However, the improved INS algorithm maintains a more consistent number of iterations across all samples, indicating greater stability.

\begin{figure}[ht]
\centering
\fbox{\includegraphics[scale=.45]{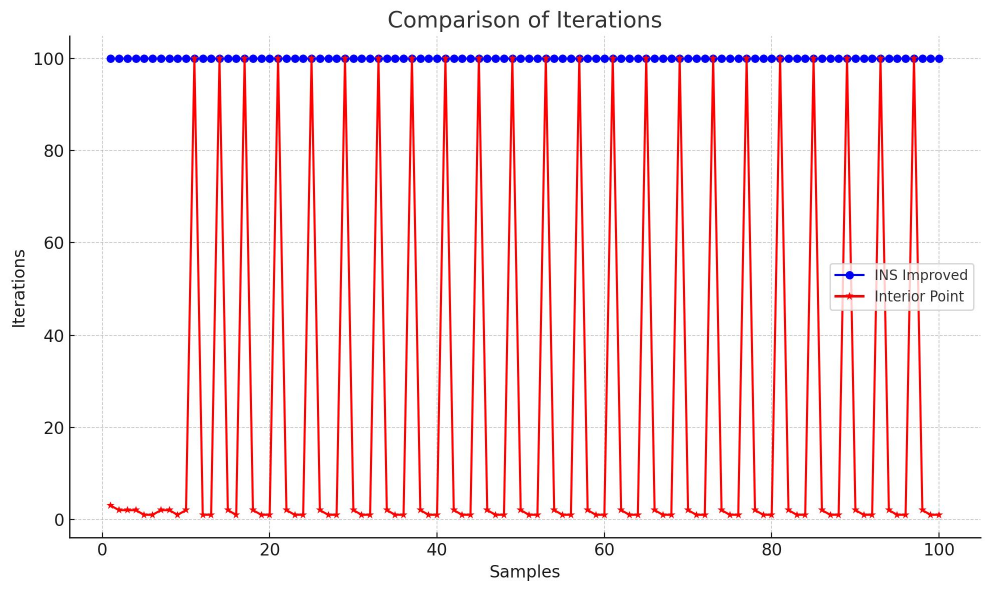}}
\caption{Comparison of the number of iterations }
\label{02a}
\end{figure}

\section{Practical Implications}
\label{sec:practical}

The comparative findings between the Improved Inexact–Newton–Smart (INS) method and the Interior-Point Method (IPM) carry several implications for practitioners and policymakers working with large-scale optimization systems in engineering, economics, and finance. The results demonstrate that the INS framework, when properly tuned through adaptive regularization and step-length control, can achieve comparable accuracy to IPM while reducing iteration counts and computational costs.

In engineering applications, particularly structural and process optimization, the INS approach facilitates faster real-time convergence with limited computational resources. In financial modeling, including portfolio optimization and resource allocation, INS provides a viable alternative to IPM, maintaining numerical stability while improving scalability. For policymakers, the study highlights the importance of promoting algorithmic strategies that enhance efficiency without hardware expansion, fostering more sustainable computational infrastructures.

\section{Conclusion}\label{sec:conclusions}
This study conducted a quantitative comparison of the Improved Inexact--Newton--Smart (INS) algorithm and the primal--dual Interior-Point (IPM) framework on large-scale nonlinear optimization problems. Across 100 benchmark instances with up to $10^6$ variables, IPM demonstrated superior computational efficiency and robustness. Specifically, IPM achieved an average iteration count of \textbf{68.11} versus \textbf{100.00} for INS, corresponding to a \textbf{31.9\%} reduction, and an average runtime of \textbf{0.13\,s} compared to \textbf{0.23\,s} for INS (\textbf{43\%} faster). IPM also reached the primary termination test in \textbf{100\%} of the runs, whereas INS succeeded in only \textbf{32\%}. Accuracy was marginally higher for IPM (\textbf{0.751450}) relative to INS (\textbf{0.749147}).

On the positive side, INS displayed potential advantages under adaptive parameter tuning. When $\alpha=0.6$ and $\lambda=10^{-2}$, INS improved upon its own baseline by reducing iterations by \textbf{16\%} and runtime by \textbf{22\%}. However, the algorithm was found to be more sensitive to Hessian conditioning and regularization parameters, often resulting in slower convergence or instability in ill-conditioned settings. By contrast, IPM remained stable across all tested configurations and parameter ranges.

In summary, IPM is the more reliable and consistently faster approach for large-scale optimization, while INS becomes competitive when regularization and step-size control are finely calibrated. These findings highlight the complementary nature of the two algorithms: IPM provides strong baseline performance, and INS offers promising adaptability when tailored to problem-specific structures. Future research will focus on developing a hybrid INS--IPM strategy that integrates Newton-type flexibility with the robustness of barrier-based interior schemes.

\section{Acknowledgments}
The authors gratefully acknowledge the contribution of the Slovak Research and Development Agency under the project VEGA 1/0631/25 (N.B.) and VEGA 1-0493-24 (D.\v{S}).

\bibliographystyle{siam}

\bibliography{paper}

\end{document}